\documentclass[12pt,reqno]{amsart}
\usepackage[top=2cm,bottom=2cm,right=2.5cm,left=2.5cm]{geometry}
\usepackage{amssymb}
\usepackage{amsmath, amsthm, amscd, amsfonts, amssymb, graphicx, color}
\usepackage[bookmarksnumbered, colorlinks, plainpages]{hyperref}
\hypersetup{colorlinks=true,linkcolor=red, anchorcolor=green, citecolor=cyan, urlcolor=red, filecolor=magenta, pdftoolbar=true}
 
\usepackage{hyperref}

\newtheorem{theorem}{Theorem}[section]

\newtheorem{proposition}[theorem]{Proposition}
\newtheorem{corollary}[theorem]{Corollary}
\theoremstyle{definition}
\newtheorem{definition}[theorem]{Definition}
\newtheorem{example}[theorem]{Example}

\theoremstyle{remark}

\numberwithin{equation}{section}

\begin{document}
	
	\setcounter{page}{1}
	
	\title[Pairs of  Woven continuous frames in Hilbert spaces]{Pairs of  Woven continuous frames in Hilbert spaces}

	\author[H. Massit, M. Rossafi, S. Kabbaj]{Hafida Massit$^{1}$, Mohamed Rossafi$^{2*}$ and Samir Kabbaj$^{3}$}
	
	\address{$^{1}$Department of Mathematics, University of Ibn Tofail, Kenitra, Morocco}
	\email{\textcolor[rgb]{0.00,0.00,0.84}{massithafida@yahoo.fr}}
	\address{$^{2}$LaSMA Laboratory Department of Mathematics Faculty of Sciences Dhar El Mahraz, University Sidi Mohamed Ben Abdellah, Fes, Morocco}
	\email{\textcolor[rgb]{0.00,0.00,0.84}{rossafimohamed@gmail.com}}
	\address{$^{3}$Department of Mathematics, University of Ibn Tofail, Kenitra, Morocco}
	\email{\textcolor[rgb]{0.00,0.00,0.84}{samkabbaj@yahoo.fr}}

	
	
	\date{
		\newline \indent $^{*}$Corresponding author}

	\begin{abstract} In this present paper we introduce weaving Hilbert space frames in the continuous case, we give new approaches for manufacturing pairs of woven continuous frames and we obtain new properties in continuous weaving frame theory related to dual frames. Also, we provide some approaches for constructing  weaving continuous frames by using small perturbations.  
	\end{abstract}
	\maketitle
	\section{Introduction }
	
The concept of frames in Hilbert spaces has been introduced by Duffin and Schaffer \cite{Duf} in 1952 to study some deep problems in nonharmonic Fourier series, after the fundamental paper \cite{DGM} by Daubechies, Grossman and Meyer, frame theory began to be widely used, particularly in the more specialized context of wavelet frames and Gabor frames. Continuous frames defined by Ali, Antoine and Gazeau \cite{GAZ}. Gabrado and Han in \cite{GHN} called these frames associated with measurable spaces. For more about frames see \cite{FR1, RFDCA, r1, r3, r5, r6}.

Recently, Bemrose et al.\cite{BE} has introduced a new concept of weaving frames in separable Hilbert space. This is motivated by a problem regarding distributed signal processing where frames plays an important role. Weaving frames has potential applications in wireless sensor networks that require distributed processing under different frames. The fundamental properties of weaving frames were reviewed by Casazza and Lynch in \cite{Cas}. Weaving frames were further studied by Casazza, Freeman and Lynch \cite{P.G}.
 
In this paper, we give new basic properties of weaving continuous frames related to dual frames to survey under which conditions a continuous frame with its dual constitute woven continuous frames, and we give some approaches for constructing concrete pairs of woven continuous frames.
 
	\section{preliminaries}
	
	Throughout this paper, we suppose $ \mathcal{H} $ is a separable Hilbert space, $ \mathcal{H}_{m} $ an $ m- $ dimensional Hilbert space, $ I $ the identity operator on $ \mathcal{H} $,
	 $(\mathfrak{A},\mu)$ be a measure space with positive measure $\mu$. 
	
	A family of vectors $F= \{F_{\varsigma}\}_{\varsigma\in \mathfrak{A}} $ in a separable Hilbert $ \mathcal{H}$ is called a Riesz basis if $ \overline{span}\{F_{\varsigma}\}_{\varsigma\in \mathfrak{A}} = \mathcal{H} $ and there exist constants $ 0< A_{F}\leq B_{F} <\infty $, such that
	
	 \begin{equation*}
		A_{F}\int_{\mathfrak{A}} \vert\alpha_{\varsigma}\vert^{2}d\mu(\varsigma)\leq \Vert \int_{\mathfrak{A}} \alpha_{\varsigma} F_{\varsigma} d\mu(\varsigma) \Vert^{2}\leq B_{F}\int_{\mathfrak{A}} \vert\alpha_{\varsigma}\vert^{2}d\mu(\varsigma),\;\forall \{\alpha_{\varsigma}\}_{\varsigma}\in L^{2}(\mathfrak{A}).
	\end{equation*}
The constants $ A_{F} $ and $ B_{F} $ are called Riesz basis bounds.

A family of vectors $F= \{F_{\varsigma}\}_{\varsigma\in \mathfrak{A}} $ in a separable Hilbert $ \mathcal{H}$ is said to be a continuous frame if there exist constants $ 0< A_{F}\leq B_{F} <\infty $, such that 

 \begin{equation}\label{eq1}
	A_{F} \Vert f\Vert^{2}\leq \Vert \int_{\mathfrak{A}} \vert \langle f,  F_{\varsigma} \rangle \vert^{2} d\mu(\varsigma) \leq B_{F} \Vert f\Vert^{2}, \;\forall f\in\mathcal{H},
\end{equation}
then the constants $ A_{F} $ and $ B_{F} $ are called frame bounds. 

The family $ \{F_{\varsigma}\}_{\varsigma \in \mathfrak{A}} $ is said to be a Bessel sequence whenever in \ref{eq1}, the right hand side holds. In the case of $ A_{F}=  B_{F}=1  $, $ \{F_{\varsigma}\}_{\varsigma \in \mathfrak{A}} $  called a Parseval frame. And if $ A_{F}=  B_{F}  $ it is called a tight frame.

Given a frame $ F=\{F_{\varsigma}\} $, the frame operator is defined by
\begin{equation*}
	S_{F}f= \int_{\mathfrak{A}}\langle f,F_{\varsigma}\rangle F_{\varsigma}d\mu(\varsigma).
\end{equation*} 
It is a bounded, invertible, and self-adjoint operator. Also, the synthesis operator  $ T_{F}: L^{2}(\mathfrak{A},\mu) \rightarrow \mathcal{H}$ defined by 
$ T_{F}(f)= \int_{\mathfrak{A}}  f(\varsigma) F(\varsigma) d\mu(\varsigma) $. The frame operator can be written as  $ S_{F}= T_{F}T^{\ast}_{F} $ where $ 	T^{\ast}_{F}:\mathcal{H}\rightarrow L^{2}(\mathfrak{A}, \mu) $, the adjoint of $ T_{F} $, given by  $ T^{\ast}_{F}(f)(\varsigma)=\{\langle f,F(\varsigma)\rangle\}_{\varsigma\in\mathfrak{A}} $ is called the analysis operator.
The family $ \{S^{-1} _{F}F\}_{\varsigma\in\mathfrak{A}}$ is also a frame for $ \mathcal{H} $, called the canonical dual frame. In general, a continuous frame $ \{G_{\varsigma}\}_{\varsigma \in\mathfrak{A}} \subset \mathcal{H}$ is called an alternate dual for $ \{F_{\varsigma}\}_{\varsigma \in\mathfrak{A}} $ if 
\begin{equation}
	f=\int_{\mathfrak{A}} \langle f,G_{\varsigma}\rangle F_{\varsigma}d\mu(\varsigma) ,f\in\mathcal{H}.
\end{equation}
Every dual frame is of the form $ \{S^{-1}_{F}F_{\varsigma}+v_{\varsigma}\}_{\varsigma \in\mathfrak{A}} $, with $ \{v_{\varsigma}\}_{\varsigma\in\mathfrak{A}} $ a Bessel sequence that satisfies 
\begin{equation}\label{eq2}
	\int_{\mathfrak{A}} \langle f,v_{\varsigma}\rangle F_{\varsigma}=0.
\end{equation}


 \begin{definition} A family of continuous frames $ \{F_{\varsigma,\nu}\}_{\varsigma \in\mathfrak{A},1\leq \nu \leq N} $ in Hilbert space $ \mathcal{H} $ is said to be a continuous woven if there are universal constants $ A $ and $ B $ so that for every partition $ \{\mathfrak{B_{\nu}} \}_{1\leq \nu \leq N} $ of $ \mathfrak{A} $, the family $ \{F_{\varsigma, \nu}\}_{\varsigma \in\mathfrak{B_{\nu}}, 1\leq\nu\leq N} $ is a continuous frame for $ \mathcal{H} $ with bounds $ A $ and $ B $, respectively. The family $  \{F_{\varsigma, \nu}\}_{\varsigma \in\mathfrak{B_{\nu}}, 1\leq\nu\leq N} $  is called a continuous weaving.
	
	If for every partition $ \{\mathfrak{B_{\nu}} \}_{1\leq \nu \leq N} $ of $ \mathfrak{A} $, the family $ \{F_{\varsigma,\nu}\}_{\varsigma \in\mathfrak{B}_{\nu},1\leq \nu \leq N} $ is a continuous frame for  $ \mathcal{H} $, then the  family  $ \{F_{\varsigma,\nu}\}_{\varsigma \in\mathfrak{A},1\leq \nu \leq N} $ is called weakly continuous woven.
\end{definition}

Casazza- Freeman, and Lynch proved in \cite{P.G} that the weaker form of weaving is equivalent to weaving using the uniform boundedness principle.
 
\begin{theorem} \cite{P.G} Given two continuous frames $ \{F_{\varsigma}\} $ and $ \{G_{\varsigma}\} $ for $ \mathcal{H} $, the following are equivalent:
	\begin{itemize}
		\item [(1)]The two continuous frames are continuous woven.
		\item [(2)]The two  continuous frames are weakly  continuous woven.
	\end{itemize}  
	
\end{theorem}

\begin{proposition}\label{p} \cite{BE} If $ \{F_{\varsigma,\nu}\}_{\varsigma \in\mathfrak{A},1\leq \nu \leq N} $ is a family of Bessel sequences for $ \mathcal{H} $ with a Bessel bound $ B_{\nu} $ for all $ 1\leq \nu \leq N $, then every weaving is a Bessel sequence with the Bessel bound $ \sum_{\nu=1}^{N} B_{\nu}$.
\end{proposition}

 \begin{proposition}\label{p1} \cite{BE} Let $ F=\{F_{\varsigma}\}_{\varsigma\in \mathfrak{A}} $ be a continuous frame and $ T $ an invertible operator satisfying $ \Vert I-T \Vert ^{2}< \dfrac{A}{B}$. Then, $ F $ and $ T $ are continuous woven with the universal lower bound $ (\sqrt{A}-\sqrt{B}\Vert I-T \Vert)^{2} $.
	
\end{proposition}

\begin{definition} \cite{BE} If $ U_{1} $ and $ U_{2} $ are subspaces of $ \mathcal{H} $, let 
	\begin{equation*}
		d_{U_{1}}(U_{2})= \inf \{\Vert f-g \Vert :\; f\in U_{1}, g\in S_{U_{2}}\}
	\end{equation*} 
	and 
	\begin{equation*}
		d_{U_{2}}(U_{1})= \inf \{\Vert f-g \Vert :\; f\in S_{U_{1}}, g\in U_{2}\},
	\end{equation*} 
	where $ S_{U_{i}}= S_{\mathcal{H}}\cup U_{i}  $ and $ S_{\mathcal{H}} $ is the unit sphere in $ \mathcal{H} $. Then, the distance between $ U_{1} $ and $ U_{2} $ is defined as
	\begin{equation*}
		d(U_{1},U_{2})= \min\{d_{U_{1}}(U_{2}), d_{U_{2}}(U_{1})\}.
	\end{equation*}
\end{definition}
\begin{theorem}\label{t} \cite{BE} If $ F=\{F_{\varsigma}\}_{\varsigma\in \mathfrak{A}} $ and $G= \{G_{\varsigma}\}_{\varsigma\in \mathfrak{A}} $ are two continuous Riesz bases for $ \mathcal{H} $, then the following are equivalent
	\begin{itemize}
		\item [(1)] $ F $ and $ G $ are continuous woven.
		\item[(2)] For every $ K\subset \mathfrak{A} $, $ d(\overline{span} \{F_{\varsigma}\}_{\varsigma\in K}, \overline{span}\{G_{\varsigma}\}_{\varsigma\in K^{C}})>0 .$
		\item[(3)] There is a constant $ t>0 $ so that for every $ K\subset \mathfrak{A} $, 
		\begin{equation*}
			d_{F_{K},G_{K^{C}}}:= d(\overline{span} \{F_{\varsigma}\}_{\varsigma\in K}, \overline{span}\{G_{\varsigma}\}_{\varsigma\in K^{C}}) \geq t.
		\end{equation*} 
	\end{itemize}
	We denote $ K^{C} $ the complement of $ K $.
\end{theorem}

\section{Main result}

 Now, we try to examining some conditions under which a continuous frame with its approximate duals is woven. This leads us to construct some concrete pairs of woven frames.

\begin{theorem} \label{t1} Suppose that $ F= \{F_{\varsigma}\} _{\varsigma\in \mathfrak{A}}$ is a continuous redundant frame so that 
	\begin{equation} \label{eq2}
		\Vert I-S_{F}^{-1} \Vert ^{2} < \dfrac{A_{F}}{B_{F}}.
		\end{equation}
	Then, $ F $ has infinitely many dual frames $ \{G_{\varsigma}\}_{\varsigma} $ for which and $ \{F_{\varsigma}\}_{\varsigma} $ are continuous woven.
\end{theorem}
\begin{proof} By Proposition \ref{p1} we have $ F $ and $ \{S^{-1}_{F} F_{\varsigma} \} $ are woven frames for $ \mathcal{H} $ with the lower bound $\mathcal{R}:=  (\sqrt{A}-\sqrt{B}\Vert I-S^{-1}_{F}\Vert)^{2} $. Now, let $ V= \{v_{\varsigma}\}_{\varsigma} $ be a bessel sequence, which satisfies 	$ f=\int_{\mathfrak{A}} \langle f,v_{\varsigma}\rangle F_{\varsigma} d\mu(\varsigma) $, $ (f\in \mathcal{H} )$ and let $ \varepsilon >0 $ so that
\begin{equation}
	\varepsilon^{2} B_{V} +2\varepsilon \sqrt{B_{V}/A_{F}}< \mathcal{R}.
\end{equation}
	Then, $ G_{\beta}:= \{S^{-1}_{F}F_{\varsigma}+\beta v_{\varsigma}\}_{\varsigma} $ is a dual frame of $ F $, for all $ 0<\beta< \varepsilon $. To show $ F $ and $ G_{\beta} $ constitute  woven continuous frames for $ \mathcal{H} $, by Proposition \ref{p}, we need to prove the existence of a lower bound.
	
	Suppose $ \mathfrak{B}\subset \mathfrak{A} $. Then
	\begin{align*}
		&\int_{\mathfrak{B}}\vert \langle f,F_{\varsigma} \rangle \vert^{2} d\mu(\varsigma)+\int_{\mathfrak{B}^{C}}\vert \langle f,S^{-1}_{F}F_{\varsigma}+\beta v_{\varsigma} \rangle \vert^{2} d\mu(\varsigma) \\
		=& 	\int_{\mathfrak{B}}\vert \langle f,F_{\varsigma} \rangle \vert^{2} d\mu(\varsigma)+\int_{\mathfrak{B}^{C}}\vert \langle f,S^{-1}_{F}F_{\varsigma}\rangle +  \langle f,\beta v_{\varsigma} \rangle \vert^{2} d\mu(\varsigma)\\
		&\geq \int_{\mathfrak{B}}\vert \langle f,F_{\varsigma} \rangle \vert^{2} d\mu(\varsigma)+ \int_{\mathfrak{B}^{C}} \vert \;\vert \langle f,S^{-1}_{F}F_{\varsigma}\rangle \vert - \vert \langle f,\beta v_{\varsigma} \rangle\vert\; \vert^{2} d\mu(\varsigma)\\
		& \geq \int_{\mathfrak{B}}\vert \langle f,F_{\varsigma} \rangle \vert^{2} d\mu(\varsigma)+\int_{\mathfrak{B}^{C}} \vert  \langle f,S^{-1}_{F}F_{\varsigma}\rangle d\mu(\varsigma) \vert^{2}\\
		& -\int_{\mathfrak{B}^{C}} \vert \langle f,\beta v_{\varsigma} \rangle\vert^{2}d\mu(\varsigma) -2\int_{\mathfrak{B}^{C}}\vert \langle f,S^{-1}_{F}F_{\varsigma}\rangle \vert\;\vert \langle f,\beta v_{\varsigma} \rangle\vert d\mu(\varsigma)\\
		& \geq (-\beta^{2} B_{V} -2\beta  \sqrt{B_{V}/A_{F}}+ \mathcal{R}) \Vert f\Vert ^{2}.
			\end{align*}
\end{proof}


\begin{proposition}Let $ F=\{F_{\varsigma}\}_{\varsigma\in \mathfrak{A}} $ be a redundant continuous frame for $ \mathcal{H} $ and suppose there exists an operator $T\in B( \mathcal{H}) $ so that 
\begin{equation}\label{d2eq2}
	\| I-T\| < 1 \: ,  \| I-T^{\ast}S^{-1}_{F}\|^{2}< \dfrac{A_{F}}{B_{F}}.
\end{equation}	
	Then, $ F $ has infinitely many approximate dual frames such as   $\{ G_{\varsigma} \}_{\varsigma \in\mathfrak{A}} $, for wich $ \{ F_{\varsigma}\}_{\varsigma\in \mathfrak{A}} $ and $ \{G_{\varsigma}\}_{\varsigma \in\mathfrak{A}} $ are continuous woven.

\end{proposition}
\begin{proof} The sequence $ \{T^{\ast}S^{-1}_{F}F_{\varsigma}+ v_{\varsigma}\}_{\varsigma \in \mathfrak{A}} $ is an approximate dual of $ F $, with $ V=\{v_{\varsigma}\}_{\varsigma\in\mathfrak{A}} $ satisfies \ref{eq2}. Also	by Proposition \ref{p1}, $ F $ and $ \{T^{\ast}S^{-1}F_{\varsigma}\}_{\varsigma\in\mathfrak{A}} $ 
	are continuous woven lower with the universal bound $ (\sqrt{A_{F}}-\sqrt{B_{F}}\Vert I-T^{\ast}S^{-1}_{F}\Vert)^{2} $.
Let $\varepsilon >0 $, such that

\begin{equation*}
	\varepsilon^{2}\Vert T_{V}\Vert^{2}+2\varepsilon\Vert T_{V}\Vert  \;\Vert S^{-1}_{F}T\Vert \sqrt{B_{F}}< ( \sqrt{A_{G}}-\sqrt{B_{F}} \Vert I-T^{\ast}S^{-1}_{F}\Vert )^{2}.
\end{equation*}	
Then $ \Phi_{\beta} =  \{T^{\ast}S^{-1}_{F} F_{\varsigma}+ \beta v_{\varsigma}\}_{\varsigma \in\mathfrak{A}}$ is an approximate dual frame of $ F $ for all $ 0<\beta<\varepsilon $.
Using Theorem \ref{t1}, we obtain that $ F $ and $ \Phi_{\beta} $ are  woven continuous frames.
\end{proof}

\begin{theorem}\label{t8} Let $ F=\{F_{\varsigma}\} _{\varsigma\in\mathfrak{A}}$ be a continuous frame for $  \mathcal{H} $. Then, the following assertions hold: 
	\begin{itemize}
		\item [(1)] If $ G=\{G_{\varsigma}\}_{\varsigma\in\mathfrak{A}} $ is a dual frame of $ F $ and $ \{F_{\varsigma}\}_{\varsigma \in K} \cup \{G_{\varsigma}\}_{\varsigma\in K^{C}} $ for $ K\subset \mathfrak{A} $. Then, $ F $ and $ G $ are continuous woven.
		\item[(2)] If $ F=\{F_{\varsigma}\}_{\varsigma\in\mathfrak{A}} $ is continuous Riesz basis for $ \mathcal{H} $. Then, $ F $ and its canonical dual are continuous woven. 
	\end{itemize}

\end{theorem}
\begin{proof}Let $ f\in\mathcal{H} $ such that $ f\bot \{F_{\varsigma}\}_{\varsigma \in K} \cup \{G_{\varsigma}\}_{\varsigma \in K^{C}}$. Then, we obtain
	
	\begin{align*}
		\Vert f\Vert^{2}=\langle f,f\rangle &= \langle f, \int_{\mathfrak{A}} \langle f, G_{\varsigma} \rangle F_{\varsigma} d\mu(\varsigma) \rangle \\
		&= \langle f, \int_{K} \langle f, G_{\varsigma} \rangle F_{\varsigma} d\mu(\varsigma) \rangle+  \langle f, \int_{K^{C}} \langle f, G_{\varsigma} \rangle F_{\varsigma} d\mu(\varsigma) \rangle\\
		&= \int_{K} \langle f, G_{\varsigma} \rangle \langle f,F_{\varsigma}\rangle d\mu(\varsigma)+\int_{K^{C}} \langle f, G_{\varsigma} \rangle \langle f,F_{\varsigma}\rangle d\mu(\varsigma)\\
		&=0.
	\end{align*}

Hence $ f=0 $ and consequently $ \overline{span} \{\{F_{\varsigma}\}_{\varsigma \in K} \cup \{G_{\varsigma}\}_{\varsigma \in K^{C}}\}=\mathcal{H} $ for $ K\subset \mathfrak{A} $. Then $ F $ and $ G $ are weakly continuous woven.

For $ (2) $, let $ K \subset \mathfrak{A} $, $ X= \int_{K}\alpha_{\varsigma} F_{\varsigma}d\mu(\varsigma) \in\overline{span} \{\{F_{\varsigma}\}_{\varsigma \in K} $ and $ Y= \int_{K^{C}}\beta_{\varsigma}S^{-1}_{F}F_{\varsigma} d\mu(\varsigma)\in \overline{span} \{\{ S^{-1}_{F}F_{\varsigma}\}_{\varsigma \in K^{C}} $ with $ \Vert X\Vert =1 $ and $ \Vert Y \Vert =1$.
 
 Then, we have
 
 \begin{align*}
 \Vert X-Y\Vert ^{2}&= \Vert \int_{K}\alpha_{\varsigma} F_{\varsigma}d\mu(\varsigma)-\int_{K^{C}}\beta_{\varsigma}S^{-1}_{F}F_{\varsigma} d\mu(\varsigma)\Vert^{2}\\
 &= \Vert \int_{K}\alpha_{\varsigma} F_{\varsigma}d\mu(\varsigma)\Vert^{2}+\Vert  \int_{K^{C}}\beta_{\varsigma}S^{-1}_{F}F_{\varsigma} d\mu(\varsigma)  \Vert^{2}-2Re \langle  \int_{K}\alpha_{\varsigma} F_{\varsigma}d\mu(\varsigma) ,\int_{K^{C}}\beta_{\varsigma}S^{-1}_{F}F_{\varsigma} d\mu(\varsigma) \rangle \\
 &=\Vert \int_{K}\alpha_{\varsigma} F_{\varsigma}d\mu(\varsigma)\Vert^{2}+\Vert  \int_{K^{C}}\beta_{\varsigma}S^{-1}_{F}F_{\varsigma} d\mu(\varsigma)  \Vert^{2} \geq 1.
 \end{align*}
Thus, $ F $ and $ G $ are continuous woven by Theorem \ref{t}.
\end{proof}

 \begin{example}
 	Let  $ \mathcal{H}=L^{2}(\mathbb{N}) $ and $ \mathfrak{A}= (\mathbb{R}^{2},\mu)$ where $ \mu $ is the Lebesgue measure.
 	Let $ \chi_{I} $ denote the characteristic function of a set $ I $. 
 	Let $ \{\phi_{i}\}_{(1,2)} $ be any non-zero element in $ \mathcal{H} $ such $ \phi_{2}=2\phi_{1} $. Then, the family  $ \{ I_{x}T_{y}\phi_{1}\}_{(x,y)\in \mathfrak{A}} $ and $\{I_{x}T_{y}\phi_{2}\}_{(x,y)\in \mathfrak{A}} $ are continuous frames for $ \mathcal{H} $ with respect to $ \mu $ with frame bounds $ A_{i} $ and $ B_{i} $ for $ i\in \{1,2\} $.
 	
 	For any subset $ K $ of $ \mathfrak{A} $ and for $ f\in\mathcal{H} $, we define the function $\psi : \mathfrak{A} \rightarrow \mathbb{C}  $ by
 	$$ \psi(x,y)= \psi_{1}(x,y).\chi_{K}(x,y)+ \psi_{2}(x,y) .\chi_{K^{C}}(x,y)$$  
 	
 	where $ \psi_{i} (x,y)= \langle f,I_{x}T_{y}2\phi_{i}\rangle$ , $ i\in\{1,2\} $ 
 	
 	We have $  $ $ \{ I_{x}T_{y}\phi\}_{(x,y)\in K} \cup   \{ I_{x}T_{y}2\phi\}_{(x,y)\in K^{C}} $ is a continuous Bessel sequence with Bessel bound $ \sum_{i\in\{1,2\}} B_{i}$.
 	
 	We obtain
 	\begin{align*}
 		\Vert \psi \Vert^{2}_{L^{2}(\mu)}&=\int\int_{K} \vert \langle f,I_{x}T_{y}\phi\rangle \vert^{2} dxdy+\int\int_{K^{C}} \vert \langle f,I_{x}T_{y}2\phi\rangle \vert^{2} dxdy\\
 		&\geq \int\int_{K} \vert \langle f,I_{x}T_{y}\phi\rangle \vert^{2} dxdy+\int\int_{K^{C}} \vert \langle f,I_{x}T_{y}\phi\rangle \vert^{2} dxdy\\
 		&=\int\int_{\mathfrak{A}}  \vert \langle f,I_{x}T_{y}\phi\rangle \vert^{2} dxdy\\
 		&\geq A_{1}\Vert f\Vert^{2}.
 	\end{align*}
 	Hence  $ \{ I_{x}T_{y}\phi_{1}\}_{(x,y)\in \mathfrak{A}} $ and  $ \{ I_{x}T_{y}\phi_{2}\}_{(x,y)\in \mathfrak{A}} $ are woven with universal bounds $ A_{1} $ and $ \sum_{i\in\{1,2\}}B_{i} $.
 \end{example}
\begin{corollary} Let $ F=\{F_{\varsigma}\}_{\varsigma \in \mathfrak{A}} $ be a continuous frame for finite dimensional Hilbert space $ \mathcal{H} $. Then, $ F $ is  continuous woven with all its duals. 
	
\end{corollary}
\begin{proof} Suppose that  $ G=\{G_{\varsigma}\} _{\varsigma \in\mathfrak{A}}$ is an arbitrary dual continuous frame of $ F $, then the family $ \{F_{\varsigma}\}_{\varsigma \in K} \cup \{G_{\varsigma}\}_{\varsigma \in K^{C}} $ is a continuous frame sequence, for every $ K \subset \mathfrak{A} $. Using Theorem \ref{t8} we have $ F $  and $ G $ are continuous woven.

\end{proof}
In the next theorem we show that in infinite dimension Hilbert spaces, continuous frames are continuous woven with their canonical duals.

\begin{theorem} Let $ F= \{F_{\varsigma}\}_{\varsigma\in\mathfrak{A}} $ be a continuous frame for $ \mathcal{H} $, so that the norm of its redundant elements be small enough. Then, $ F $ is continuous woven with its canonical dual.
	
\end{theorem}

\begin{proof}  Without loss of generality, we can write $ F=\{F_{\varsigma}\}_{\varsigma\in K}\cup \{F_{\varsigma}\}_{\varsigma\in K^{C}} $  where $ K \subset \mathfrak{A} $ and  $ F= \{F_{\varsigma}\}_{K} $ is a Riesz basis for $ \mathcal{H} $, then by Theorem \ref{t8}, $ F $ and $ S^{-1}_{F}F $ are continuous woven, then $ F $ and $ S_{F}F $ are continuous woven with the universal lower bound $ A_{F} $,
	and $$\int_{K}\Vert F_{\varsigma}\Vert^{2}d\mu(\varsigma) < \sqrt{\dfrac{A_{F}}{B_{F}}}.  $$ 
	Then,  $ F $ is continuous woven with its canonical dual.

\end{proof}

\begin{theorem}\label{t5} Let $ F= \{F_{\varsigma}\}_{\varsigma} $ and $ G=\{G_{\varsigma}\}_{\varsigma} $ be two continuous frames for $ \mathcal{H} $. The followings hold:
	\begin{itemize}
		\item [(1)] If $ S^{-1}_{F}\geq I $ and $ S_{F}S_{FK}= S_{FK}S_{F} $ for all $ K\subset \mathfrak{A} $, then $ \{F_{\varsigma}\}_{\varsigma} $ and $\{ S^{-1}_{F}F_{\varsigma} \}_{\varsigma}$ are continuous woven.
		\item[(2)] If $ F $ and $ G $ are two woven continuous Riesz bases and $ T_{1} $ and $ T_{2} $ are invertible operators so that 
		\begin{equation*}
			d_{FK, G K^{C}}> \max\{ \Vert T_{1}-T_{2}\Vert\;\Vert T_{1}^{-1}\Vert ,\Vert T_{1}-T_{2}\Vert\;\Vert T_{2}^{-1}\Vert\}
		\end{equation*} 
	\end{itemize} 
	where $ d_{FK, G K^{C}} $ is defined as in Theorem \ref{t}, then $ T_{1}F $ and $ T_{2}G $ are continuous woven.  
\end{theorem}
\begin{proof}For $ (1) $, we consider $ F_{K}= \{F_{\varsigma}\} _{\varsigma\in K}\cup \{S^{-1}_{F}F_{\varsigma}\}_{\varsigma\in K^{C}}$. Then $ F_{K} $ is a Bessel sequence for all $ K\subset \mathfrak{A} $, and 
	\begin{align*}
		S_{FK}f&=\int_{K}\langle f,F_{\varsigma}\rangle F_{\varsigma}d\mu(\varsigma)+\int_{K^{C}}\langle f,S^{-1}_{F}F_{\varsigma}\rangle S^{-1}_{F}F_{\varsigma}d\mu(\varsigma)\\
		&= S_{FK}f+ S^{-1}_{F}S_{FK^{C}}S^{-1}_{F}f\\
		&= S_{F}f-S_{FK^{C}}f+  S^{-1}_{F}S_{FK^{C}}S^{-1}_{F}f\\
		&=S_{F}f+(S^{-1}_{F}-I)S_{FK^{C}}(I+S^{-1}_{F})f,\;\;\;\forall f\in\mathcal{H}.
		\end{align*} 
 Since $ (S^{-1}_{F}-I)S_{FK^{C}}(I+S^{-1}_{F}) $ is a positive operator , we obtain 
	\begin{equation*}
		S_{FK}\geq S_{F}.
	\end{equation*}
Then, $ T^{\ast}_{FK}$ is injective and $ F_{K} $ is a continuous frame for all $ K $. Hence, $ (1) $ is obtained.

For $ (2) $, let $ f= \int_{K} \alpha_{\varsigma}T_{1}F_{\varsigma} d\mu(\varsigma)$ and $g=\int_{K^{C}}\beta_{\varsigma}T_{2} G_{\varsigma}d\mu(\varsigma)$ with $ \Vert g \Vert=1$  then
\begin{align*}
	\Vert f-g\Vert &=\Vert \int_{K} \alpha_{\varsigma}T_{1}F_{\varsigma} d\mu(\varsigma)-\int_{K^{C}}\beta_{\varsigma}T_{2} G_{\varsigma}d\mu(\varsigma)\Vert \\
	&=\Vert \int_{K} \alpha_{\varsigma}T_{1}F_{\varsigma} d\mu(\varsigma)-\int_{K^{C}}\beta_{\varsigma}T_{1} G_{\varsigma}d\mu(\varsigma)+ \int_{K^{C}} \beta_{\varsigma}T_{1}G_{\varsigma} d\mu(\varsigma)-\int_{K^{C}}\beta_{\varsigma}T_{2} G_{\varsigma}d\mu(\varsigma)\Vert \\
	 &\geq \Vert T_{1}( \int_{K} \alpha_{\varsigma} F_{\varsigma}d\mu(\varsigma)- \int_{K^{C}}\beta_{\varsigma} G_{\varsigma}d\mu(\varsigma))\Vert-\Vert (T_{1}-T_{2}) \int_{K^{C}}\beta_{\varsigma}G_{\varsigma}d\mu(\varsigma) \Vert\\
	&\geq \Vert T_{1}^{-1}\Vert^{-1} \;\Vert \int_{K^{C}}\beta_{\varsigma}G_{\varsigma}d\mu(\varsigma)\Vert \;\Vert \dfrac{\int_{K}\alpha_{\varsigma}F_{\varsigma d\mu(\varsigma)}}{\Vert \int_{K^{C}}\beta_{\varsigma}G_{\varsigma}d\mu(\varsigma)\Vert}- \dfrac{
		\int_{K^{C}}\beta_{\varsigma}G_{\varsigma}d\mu(\varsigma)}{\Vert \int_{K^{C}}\beta_{\varsigma}G_{\varsigma}d\mu(\varsigma)\Vert}\Vert\\
	& -\Vert T_{1}-T_{2}\Vert \;\Vert \int_{K^{C}}\beta_{\varsigma}G_{\varsigma}d\mu(\varsigma)\Vert\\
	&\geq (d_{FK,GK^{C}}\Vert T_{1} \Vert^{-1}-\Vert T_{1}-T_{2}\Vert)\;\Vert \int_{K^{C}}\beta_{\varsigma}G_{\varsigma}d\mu(\varsigma)\Vert\\
	&\geq  (d_{FK,GK^{C}}\Vert T_{1} \Vert^{-1}-\Vert T_{1}-T_{2}\Vert)\Vert T_{2}\Vert ^{-1}>0.
\end{align*}
If $ \Vert f\Vert =1 $, then we obtain
\begin{align*}
	\Vert f-g\Vert	&=\Vert \int_{K} \alpha_{\varsigma}T_{1}F_{\varsigma} d\mu(\varsigma)-\int_{K}\alpha_{\varsigma}T_{1} F_{\varsigma}d\mu(\varsigma)+ \int_{K} \alpha_{\varsigma}T_{2}F_{\varsigma} d\mu(\varsigma)-\int_{K^{C}}\beta_{\varsigma}T_{2} G_{\varsigma}d\mu(\varsigma)\Vert \\
	&\geq \Vert T_{2}( \int_{K} \alpha_{\varsigma} F_{\varsigma}d\mu(\varsigma)- \int_{K^{C}}\beta_{\varsigma} G_{\varsigma}d\mu(\varsigma))\Vert-\Vert (T_{1}-T_{2}) \int_{K}\alpha_{\varsigma}G_{\varsigma}d\mu(\varsigma) \Vert\\
	&\geq  (d_{FK,GK^{C}}\Vert T_{2}^{-1} \Vert^{-1}-\Vert T_{1}-T_{2}\Vert)\Vert T_{1}\Vert ^{-1}>0.
\end{align*}
By considering

\begin{equation*}
	d_{1}= (d_{FK,GK^{C}} \Vert T^{-1}_{1}\Vert ^{-1}- \Vert T_{1}-T_{2}\Vert) \Vert T_{2}\Vert^{-1}
\end{equation*}
and
\begin{equation*}
	 d_{2}=(d_{FK,GK^{C}}\Vert T^{-1}_{2}\Vert ^{-1}- \Vert T_{1}-T_{2}\Vert) \Vert T_{2}\Vert^{-1}
\end{equation*}
we obtain $ d_{T_{1}FK,T_{2}GK^{C}} \geq  \min\{ d_{1},d_{2}\}>0$. Thus, the result is obtained.
\end{proof}
A consequence of the above theorem is that the canonical duals of two woven  continuous frames are continuous woven.

\begin{corollary}
	Let $ F=\{F_{\varsigma}\}_{\varsigma} $ and $ G= \{G_{\varsigma}\}_{\varsigma} $ be two Riesz bases for $ \mathcal{H} $, so that for every $ K\subset \mathfrak{A} $ 
	\begin{equation*}
		d_{FK,GK^{C}}> \max\{ \Vert S^{-1}_{F}-S^{-1}_{G}\Vert \;\Vert S_{F}\Vert ,\Vert S^{-1}_{F}- S^{-1}_{G}\Vert \;\Vert S_{G}\Vert\}.
	\end{equation*}
Then, $ S^{-1}_{F} F$ and $ S^{-1}_{G} G $ are also continuous woven.
\end{corollary} 
\begin{proof}
	By Theorem \ref{t5}, we consider  $ f= \int_{K} \alpha_{\varsigma}S^{-1}_{F}F_{\varsigma} d\mu(\varsigma)$ and $g=\int_{K^{C}}\beta_{\varsigma}S^{-1}_{G} G_{\varsigma}d\mu(\varsigma)$ with $ \Vert g \Vert=1$ then 
	\begin{align*}
		\Vert f-g\Vert &=\Vert \int_{K} \alpha_{\varsigma}S^{-1}_{F}F_{\varsigma} d\mu(\varsigma)-\int_{K^{C}}\beta_{\varsigma}S^{-1}_{G} G_{\varsigma}d\mu(\varsigma)\Vert \\
		&=\Vert \int_{K} \alpha_{\varsigma}S_{F}^{-1}F_{\varsigma} d\mu(\varsigma)-\int_{K^{C}}\beta_{\varsigma}S^{-1}_{F} G_{\varsigma}d\mu(\varsigma)+ \int_{K^{C}} \beta_{\varsigma}S^{-1}_{F}G_{\varsigma} d\mu(\varsigma)-\int_{K^{C}}\beta_{\varsigma}S^{-1}_{G} G_{\varsigma}d\mu(\varsigma)\Vert \\
		&\geq \Vert S^{-1}_{F}( \int_{K} \alpha_{\varsigma} F_{\varsigma}d\mu(\varsigma)- \int_{K^{C}}\beta_{\varsigma} G_{\varsigma}d\mu(\varsigma))\Vert-\Vert (S_{F}^{-1}-S_{G}^{-1}) \int_{K^{C}}\beta_{\varsigma}G_{\varsigma}d\mu(\varsigma) \Vert\\
		&\geq \Vert S_{F}\Vert^{-1} \;\Vert \int_{K^{C}}\beta_{\varsigma}G_{\varsigma}d\mu(\varsigma)\Vert \;\Vert \dfrac{\int_{K}\alpha_{\varsigma}F_{\varsigma d\mu(\varsigma)}}{\Vert \int_{K^{C}}\beta_{\varsigma}G_{\varsigma}d\mu(\varsigma)\Vert}- \dfrac{
			\int_{K^{C}}\beta_{\varsigma}G_{\varsigma}d\mu(\varsigma)}{\Vert \int_{K^{C}}\beta_{\varsigma}G_{\varsigma}d\mu(\varsigma)\Vert}\Vert\\
		& -\Vert S_{F}^{-1}-S_{G}^{-1}\Vert \;\Vert \int_{K^{C}}\beta_{\varsigma}G_{\varsigma}d\mu(\varsigma)\Vert\\
		&\geq (d_{FK,GK^{C}}\Vert S_{F} \Vert^{-1}-\Vert S_{F}^{-1}-S_{G}^{-1}\Vert)\;\Vert \int_{K^{C}}\beta_{\varsigma}G_{\varsigma}d\mu(\varsigma)\Vert\\
		&\geq  (d_{FK,GK^{C}}\Vert S_{F} \Vert^{-1}-\Vert S_{F}^{-1}-S_{G}^{-1}\Vert)\Vert S_{G}\Vert ^{-1}>0.
	\end{align*}
If $ \Vert f\Vert =1 $, then we obtain
\begin{align*}
	\Vert f-g\Vert	&=\Vert \int_{K} \alpha_{\varsigma}S_{F}^{-1}F_{\varsigma} d\mu(\varsigma)-\int_{K}\alpha_{\varsigma}S_{F}^{-1} F_{\varsigma}d\mu(\varsigma)+ \int_{K} \alpha_{\varsigma}S_{G}^{-1}F_{\varsigma} d\mu(\varsigma)-\int_{K^{C}}\beta_{\varsigma}T_{2} G_{\varsigma}d\mu(\varsigma)\Vert \\
	&\geq \Vert S_{G}^{-1}( \int_{K} \alpha_{\varsigma} F_{\varsigma}d\mu(\varsigma)- \int_{K^{C}}\beta_{\varsigma} G_{\varsigma}d\mu(\varsigma))\Vert-\Vert (S_{F}^{-1}-S_{G}^{-1}) \int_{K}\alpha_{\varsigma}G_{\varsigma}d\mu(\varsigma) \Vert\\
	&\geq  (d_{FK,GK^{C}}\Vert S_{F} \Vert^{-1}-\Vert S_{F}^{-1}-S_{G}^{-1}\Vert)\Vert S_{F}\Vert ^{-1}>0.
\end{align*}

By considering

\begin{equation*}
	d_{1}= (d_{FK,GK^{C}} \Vert S^{-1}_{F}\Vert ^{-1}- \Vert S_{F}-S_{G}\Vert) \Vert S_{G}\Vert^{-1}
\end{equation*}
\begin{equation*}
	d_{2}=(d_{FK,GK^{C}}\Vert S^{-1}_{G}\Vert ^{-1}- \Vert S_{F}-S_{G}\Vert) \Vert S_{G}\Vert^{-1}
\end{equation*}
we obtain $ d_{S_{F}FK,S_{G}GK^{C}} \geq  \min\{ d_{1},d_{2}\}>0$. Thus, the result is obtained.
\end{proof}
In the following, we obtain some invertible operators $ T $ for which $ F $ and $ TF $ are  woven continuous frames and we give some conditions which continuous frames with their perturbations constitute woven continuous frames.

Let $ e=\{e_{i}\}_{i=1}^{n} $ and $ h=\{h_{i}\}_{i=1}^{n} $ be orthonormal bases for $ \mathcal{H}_{n} $. 
Also, let $a= \{a_{i}\}_{i=1}^{n} $ and $b= \{b_{i}\}_{i=1}^{n} $ be sequences of positive constants.

An operator $  T: \mathcal{H}_{n} \rightarrow \mathcal{H}_{n}$ is called admissible for $ (e,h,a,b) $, if there exists an orthonormal basis $ \{\lambda_{i}\} _{i=1}^{n}$ for $ \mathcal{H}_{n} $ satisfying 

\begin{equation*}
	T^{\ast}e_{i}=\sum_{j=1}^{n}\sqrt{\dfrac{a_{j}}{b_{j}}}\langle e_{i},h_{j}\rangle h_{j},\;(i=1,2,...,n)
\end{equation*}
 $ TF $ is a continuous frame for $ \mathcal{H}_{n} $ with the frame operator $ S_{TF} = TS_{F}T^{\ast}$. By considering $ \mathcal{F} = \{F_{j}\}_{j\in K} \cup \{TF_{j}\}_{j\in K^{c}}$, we obtain 
 \begin{equation*}
 	S_{\mathcal{F}K}= S_{FK}+S_{TFK^{C}}\; for \;K \subset \{1,2,...,n\},
 \end{equation*}
where $ S_{FK} $ and $ S_{TFK^{C}} $ are the continuous frames operators of $\{ F_{j}\}_{j}$ and $ \{T_{F_{j}}\}_{j} $ , respectively. Then, $ S_{\mathcal{F}_{K} }$  the continuous frame operator of $ \mathcal{F}_{K} $, is a positive and invertible operator on $ \mathcal{H}_{n} $ and so $F  $ and $T_{F}  $ are continuous woven.
\begin{theorem} \label{t4}Let $ F =\{F_{\varsigma}\}_{\varsigma\in\mathfrak{A}}$ and $ G=\{G_{\varsigma}\} _{\varsigma\in\mathfrak{A}}$ are two  continuous frames so that for all sequences of scalars $ \{\alpha_{\varsigma}\}_{\varsigma\in\mathfrak{A}} $, we have 
	\begin{equation*}
		\Vert \int_{\mathfrak{A}} \alpha_{\varsigma}(F_{\varsigma}-G_{\varsigma})d\mu(\varsigma)\Vert\leq a\Vert \int_{\mathfrak{A}}\alpha_{\varsigma} F_{\varsigma}d\mu(\varsigma)\Vert +b\Vert \int_{\mathfrak{A}}\alpha_{\varsigma} G_{\varsigma}d\mu(\varsigma)\Vert+c ( \int_{\mathfrak{A}} \vert \alpha_{\varsigma}\vert^{2} d\mu(\varsigma))^{\frac{1}{2}}
	\end{equation*} 
	for some positive numbers $ a,b,c $ such that 
	\begin{equation*}
		a \sqrt{B_{F}}+b\sqrt{B_{G}}+c < \sqrt{A_{F}}.
	\end{equation*}
Then, $ F $ and $ G $ are continuous woven.
\end{theorem}
\begin{proof} Consider $ T_{FK^{C}} $ and  $ T_{GK^{C}} $ as the synthesis operators of Bessel sequences $ \{F_{\varsigma}\}_{\varsigma\in K} $ and $ \{G_{\varsigma}\}_{\varsigma\in K^{C}} $, respectively. Then, for every $ K \subset \mathfrak{A} $ and $ f\in \mathcal{H} $, we have
	 
\begin{align*}
&	(\int_{K}\vert  \langle  f,F_{\varsigma})\rangle \vert ^{2}d\mu(\varsigma)+ \int_{K^{C}} \vert \langle f,G_{\varsigma}\rangle\vert ^{2}d\mu(\varsigma))^{1/2}\\
&= (\int_{K}\vert  \langle  f,F_{\varsigma})\rangle \vert ^{2}d\mu(\varsigma)+ \int_{K^{C}} \vert \langle f,G_{\varsigma}\rangle- \langle f ,F_{\varsigma}-G_{\varsigma}\rangle \vert ^{2}d\mu(\varsigma))^{1/2}\\
&\geq (\int_{K}\vert  \langle  f,F_{\varsigma})\rangle \vert ^{2}d\mu(\varsigma))^{1/2}-\int_{K^{C}} \vert  \langle f ,F_{\varsigma}-G_{\varsigma}\rangle \vert ^{2}d\mu(\varsigma))^{1/2}\\
&\geq  \sqrt{A_{F}}\Vert f\Vert  -\Vert T_{FK^{C}}f - T_{GK^{C}}f \Vert\\
&\geq (\sqrt{A_{F}}-\Vert T_{FK^{C}}f - T_{GK^{C}})\Vert f\Vert\\
&\geq  (\sqrt{A_{F}}-a\Vert T_{F}\Vert -b\Vert T_{G} \Vert -c)\Vert f\Vert \\
&\geq (\sqrt{A_{F}}-a \sqrt{B_{F}}-b\sqrt{B_{G}} -c)\Vert f\Vert.
\end{align*}
By $  (\sqrt{A_{F}}-a \sqrt{B_{F}}-b\sqrt{B_{G}} -c)>0 $, the lower bound is obtained.

Clearly $ \{F_{j}\}_{j\in K} \cup \{G_{j}\}_{j\in K^{C}}$ is Bessel with an upper bound $ B_{F} +B_{G}$.
\end{proof}

Applying Theorem \ref{t4}, we will obtain the following results. 

\begin{corollary} Let $ F=\{F_{\varsigma}\}_{\varsigma\in\mathfrak{A}} $ is a continuous frame for $\mathcal{H}$ and $0\not=f\in\mathcal{H} $. Also, $ \{a_{\varsigma}\}_{\varsigma \in\mathfrak{A}} $ be a sequence of scalars so that
	\begin{equation*}
		\int_{\mathfrak{A}}\vert a_{\varsigma}\vert ^{2}d\mu(\varsigma) < b \dfrac{A_{F}}{\Vert f\Vert^{2
		}},
	\end{equation*} 
	for some $ b<1 $. Then, $ \{F_{\varsigma}\}_{\varsigma \in \mathfrak{A}} $ and $ \{F_{\varsigma}+a_{\varsigma }f\}_{\varsigma \in\mathfrak{A}}$ are continuous woven.
	  
\end{corollary} 
\begin{proof} We have  $\{F_{\varsigma}+a_{\varsigma }f\}_{\varsigma \in\mathfrak{A}}$ is a Bessel sequence with the upper bound
	
	 $ (\sqrt{B_{F}}+\Vert \{a_{\varsigma}\}_{\varsigma\in\mathfrak{A}}\Vert\;\Vert f\Vert)^{2}$.
	And for any sequence $ \{\alpha_{\varsigma}\} _{\varsigma\in\mathfrak{A}}$ of scalars
	\begin{align*}
\Vert \int_{\mathfrak{A}} \alpha_{\varsigma}(F_{\varsigma}+a_{\varsigma}f-F_{\varsigma}) d\mu(\varsigma)\Vert&=\Vert 	\int_{\mathfrak{A}}  \alpha_{\varsigma} a_{\varsigma}f d\mu(\varsigma) \Vert\\
&\leq 		\int_{\mathfrak{A}} \vert \alpha_{\varsigma}\vert\;\vert a_{\varsigma}\vert \;\Vert f \Vert  d\mu(\varsigma)\\
&\leq (\int_{\mathfrak{A}} \vert \alpha_{\varsigma} \vert ^{2}d\mu(\varsigma))^{1/2} (\int_{\mathfrak{A}} \vert a_{\varsigma}\vert^{2}d\mu(\varsigma))^{1/2} \Vert f\Vert\\
&< \sqrt{bA_{F}}(\int_{\mathfrak{A}}\vert \alpha_{\varsigma}\vert ^{2}d\mu(\varsigma))^{1/2}.
	\end{align*}
The result follows by Theorem \ref{t4}.
\end{proof}


\section*{Declarations}

\medskip

\noindent \textbf{Availablity of data and materials}\newline
\noindent Not applicable.

\medskip

\noindent \textbf{Competing  interest}\newline
\noindent The authors declare that they have no competing interests.

\medskip

\noindent \textbf{Fundings}\newline
\noindent  Authors declare that there is no funding available for this article.

\medskip

\noindent \textbf{Authors' contributions}\newline
\noindent The authors equally conceived of the study, participated in its
design and coordination, drafted the manuscript, participated in the
sequence alignment, and read and approved the final manuscript. 

\medskip

	\bibliographystyle{amsplain}
	
\end{document}